\documentclass[fms]{cuparticle}
\usepackage{graphicx}
\usepackage{amsmath,amssymb,amsthm,stmaryrd}
\usepackage[all]{xy}
\usepackage{hyperref}
\xyoption{arc}

\numberwithin{equation}{section}
\numberwithin{table}{section}
\numberwithin{figure}{section}

\newtheorem{thm}{Theorem}[section]
\newtheorem{cor}[thm]{Corollary}
\newtheorem{prop}[thm]{Proposition}

\theoremstyle{definition}
\newtheorem{defn}[thm]{Definition}
\theoremstyle{remark}
\newtheorem{rmk}[thm]{Remark}

\def\co{\colon\thinspace}
\newcommand{\mb}[1]{\mathbb{#1}}

\newcommand{\Hom}{\ensuremath{{\rm Hom}}}

\newcommand{\overto}{\mathop\rightarrow}

\newcommand{\longoverto}{\mathop{\longrightarrow}}

\newcommand{\GL}{{\rm GL}}

\newcommand{\TAF}{{\rm TAF}}

\newcommand{\Spec}{{\rm Spec}}

\newcommand{\TMF}{{\rm TMF}}

\newcommand{\Tr}{{\rm Tr}}
\newcommand{\Nm}{{\rm N}}

\newcommand{\comp}[1]{\ensuremath{#1^\wedge}}

\newcommand{\pow}[1]{\left\llbracket{#1}\right\rrbracket}

\volume{??}
\doi{??}

\begin{document}
\authorheadline{T. Lawson}
\runningtitle{Shimura curve of disc. 15 and topological automorphic forms}

\begin{frontmatter}

\title{The Shimura curve of discriminant 15 and topological
  automorphic forms}

\author[a]{Tyler Lawson}
\ead{tlawson@math.umn.edu}

\address[a]{Department of Mathematics, University of Minnesota, 206 Church Street SE, Minneapolis, MN, USA 55414}

\received{??}
\accepted{??}

\begin{abstract}
  We find defining equations for the Shimura curve of discriminant
  $15$ over $\mb Z[1/15]$.  We then determine the graded ring of
  automorphic forms over the $2$-adic integers, as well as the higher
  cohomology.  We apply this to calculate the homotopy groups of a
  spectrum of ``topological automorphic forms'' associated to this
  curve, as well as one associated to a quotient by an Atkin-Lehner
  involution.
\end{abstract}

\end{frontmatter}

\section{Introduction}

A generalized cohomology theory $E$ associates to each space $X$ a
sequence of abelian groups $E^n(X)$, often equipped with extra
structure such as a graded multiplication. These are required to
satisfy the Eilenberg-Steenrod axioms, and this alone implies that
there is a natural Atiyah-Hirzebruch spectral sequence
\[
H^s(X; E^t(\ast)) \Rightarrow E^{s+t}(X)
\]
with reasonable convergence properties when $X$ has the homotopy
type of a CW-complex.

In a coarse sense, this tells us that $E^*(X)$ combines the
cohomological data of $X$ and the $E$-cohomology of a point in some
way. However, the spectral sequence is the first step in an iterative
calculation, and much deeper information is required to complete it:
in particular, we need to find differentials and solve hidden
extensions in this spectral sequence. In a local sense these
depend on $X$: if $X$ is a CW-complex, this information is determined
by $X$ and the attaching maps for the cells of $X$. In a global
sense these depend on $E$: this information is determined by an
intricate web of connective tissue between the groups $E^t(\ast)$,
expressible in terms of cohomology operations, secondary operations,
and so on.

It has been an active research topic to determine all the data
necessary to take a graded abelian group, or a graded ring, and lift
it to a generalized cohomology theory. Quillen showed that for many
multiplicative cohomology theories there is a theory of Chern classes,
and the formula for the first Chern class of a tensor product of line
bundles gives $E^*$ a natural formal group law
\cite{quillen-fgl}. Landweber showed that the converse holds in many
circumstances: given a graded-commutative ring with an appropriate
type of graded formal group law, one can {\em realize} it by a
generalized cohomology theory \cite{landweber-exact}.

More recently, Lurie announced a very strongly functorial result of
this converse type, generalizing work of Hopkins, Miller, and
others. It requires more input: instead of just a formal group law,
Lurie's result requires an extension of it to a $p$-divisible group
satisfying a version of the Serre-Tate property
\cite{goerss-landweber-families}. This has been exploited to construct
generalized cohomology theories attached to certain moduli of abelian
varieties, under the general header of topological automorphic forms
\cite{taf}. Lurie's theorem takes a scheme (or stack) with such a
$p$-divisible group and equips its \'etale site with a sheaf of
$E_\infty$ ring spectra (the homotopy-theoretic analogue of a sheaf of
commutative differential graded algebras).

This gives us an abundance of new objects in homotopy theory, with each
described by purely algebro-geometric data. To bring ourselves back
down to earth, we must understand the consequences of what we have
done.  The data itself determines many cohomology theories $E$ and
their coefficient rings; about the connective tissue it gives not much
direct information.

Shimura curves have provided an interesting test case. These
parametrize $2$-dimensional abelian varieties with an action of a
quaternion algebra, and share many formal similarities with the moduli
of elliptic curves. They have been harder to cohomologically analyze
than the moduli of elliptic curves, but in many cases their images in
homotopy theory have provided similar answers. In \cite{shimc}, the
resulting cohomology theories were analyzed for the Shimura curves of
discriminants $6$, $10$, and $14$. The most mysterious aspect appeared
with the curve of discriminant $10$. This curve has a very different
geometry than the moduli of elliptic curves, but after $3$-adic
completion the Shimura curve and the modular curve have associated
cohomology theories that behave in an identical fashion so far as
investigations have revealed.

The program of this paper is to study the Shimura curve of
discriminant $15$ and its relation to homotopy theory.
The prime $2$ does not divide $15$, and as a result this is the first
Shimura curve where $2$-primary information in stable homotopy theory
can be extracted. The path we will follow is similar to that in
\cite{shimc}: we describe a moduli object, determine its cohomology,
and find the coefficient ring of the resulting cohomology theory by a
spectral sequence calculation.

This results in new questions. The final calculations in homotopy
theory are very similar to the calculations of Mahowald--Rezk for
topological modular forms with level $\Gamma_0(3)$ structure
\cite{mahowald-rezk-level3}---so similar that Figure~\ref{fig:e3-page}
and Figure~\ref{fig:e7-page} could easily be used as references for
the calculations that appeared in their paper. At the prime $2$, the
coefficient ring for the Shimura curve breaks up as isomorphic to a
direct sum of two pieces: a coefficient ring for topological modular
forms with level $\Gamma_0(3)$ structure, and a certain module over
it. Moreover, the Shimura curve has an action of $\mb Z/2$, and the
quotient corresponds to a cohomology theory with a new coefficient
ring: now an {\em extension} of the coefficient ring for topological
modular forms with level $\Gamma_0(3)$ structure. This, again, occurs
despite the lack of obvious geometric connection between these
moduli. It is not clear if something fundamental is guiding these
connections, or if these are merely coincidences in low degrees.

We will now give a detailed outline, beginning with a more exact
description of the objects under consideration.

There is a 4-dimensional division algebra $D$ over $\mb Q$, generated
by elements $x$ and $y$ satisfying 
\begin{equation}
  \label{eq:divisiongens}
  x^2 = -3, y^2 = 5, xy = -yx.
\end{equation}
This is uniquely characterized by the requirement that $D \otimes \mb
Q_p \cong M_2(\mb Q_p)$ precisely for primes $p \neq 3,5$ (in other
words, it has discriminant $15$).  In particular, there is an isomorphism
\[
\tau\co D \otimes \mb R \cong M_2(\mb R).
\]
On $D$ there is a reduced norm $N(a + bx + cy + dxy) = a^2 + 3b^2 -
5c^2 -15d^2$ which is multiplicative, and under $\tau$ it corresponds
to the determinant.

Within $D$ there is also a subring
\begin{equation}
  \label{eq:ordergens}
\Lambda \cong \mb Z\langle \omega,y\rangle / 
(\omega^2 + \omega + 1, y\omega = \omega^2 y),
\end{equation}
generated by $y$ and $\omega = \frac{-1+x}{2}$.  The ring $\Lambda$ is
maximal among finitely generated submodules closed under the
multiplication (a maximal order), and any other such subring is
conjugate to $\Lambda$ \cite{eichler-idealclasses}.  We
obtain an embedding of the norm-1 subgroup:
\[
\tau\co \Lambda^{N=1} \to SL_2(\mb R)
\]
This gives the action of the norm-1 elements of $\Lambda$ on the
complex upper half-plane ${\cal H}$.  The quotient orbifold ${\cal H}
/ \Lambda^{N=1} = {\cal X}^D_{\mb C}$ is called the complex Shimura
curve of discriminant $15$, and it is a parametrizing object for
$2$-dimensional complex abelian varieties $A$ equipped with an action
of $\Lambda$ (sometimes called ``fake elliptic curves'').  The
orbifold structure reflects the fact that such objects $A$ often
possess automorphisms.

The object ${\cal X}^D_{\mb C}$, through this interpretation as a
parametrizing object, has an algebraic lift.  Over $\mb Z[1/15]$,
there is a {\em stack} ${\cal X}^D$ such that maps $S \to {\cal X}^D$
parametrize 2-dimensional abelian schemes $A/S$ with an action of
$\Lambda$.  There is an underlying coarse moduli scheme $X^D$ which is
a smooth curve over $\mb Z[1/15]$ \cite{morita-shimuracurvereduction}.
The first goal of this paper is purely algebraic: it is to determine
defining equations for $X^D$.  As in \cite{shimc}, this builds on
previous work of Kurihara~\cite{kurihara-equations} and
Elkies~\cite{elkies-computations}.

The second goal concerns automorphic forms, and it requires us to
extend from $\mb Z[1/15]$ to the ring $\mb Z_2$ (though statements
could be made over general rings $R$ such that $\Lambda \otimes R
\cong M_2(R)$).  Letting ${\cal A}/{\cal X}^D$ be the universal
abelian scheme, the action of $\Lambda \otimes \mb Z_2 \cong M_2(\mb
Z_2)$ on the $2$-dimensional relative cotangent space of ${\cal
  A}_{\mb Z_2}$ at the identity splits it into two isomorphic
1-dimensional summands.  This summand is a line bundle $\omega$ on
${\cal X}^D_{\mb Z_2}$, and the sections of $\omega^{\otimes t}$ are
automorphic forms of weight $t$ on ${\cal X}^D_{\mb Z_2}$.  The second
goal of this paper is also algebraic: it is to determine the
cohomology groups $H^s({\cal X}^D_{\mb Z_2}; \omega^{\otimes t})$.  In
particular, when $s=0$ this is a graded ring of automorphic forms over
$\mb Z_2$.

We next $2$-adically complete and study $\comp{({\cal X}^D)}_2$, a
formal parameter object living over $\mb Z_2$.  In this case, the
actions of $\Lambda$ on the $2$-dimensional formal group
$\widehat{\cal A}$ and the $2$-dimensional $2$-divisible group ${\cal
  A}[2^\infty]$ factor through $\Lambda \otimes \mb Z_2$.  This splits
the $2$-divisible group into two isomorphic $1$-dimensional summands,
and similarly for the formal group.  The aforementioned theorem of
Lurie then lifts this formal group data to a derived structure sheaf
${\cal O}^{der}$ of $E_\infty$ ring spectra (see
\cite[Section~2.6]{shimc}), and the homotopy groups are determined by
a cohomology spectral sequence
\[
H^s(U; \omega^{\otimes t}) \rightarrow \pi_{2t-s} \Gamma(U,{\cal
  O}^{der}).
\]
In particular, we can define $\TAF^D = \Gamma(\comp{({\cal X}^D)}_2,
{\cal O}^{der})$ as a global section object.  The third goal of this
paper is to determine the homotopy groups of this spectrum.  As
stated, the calculations are similar to those of Mahowald-Rezk.  The
main difference is that an extra summand occurs in the computation
(Figure~\ref{fig:e3-page-a6} and Figure~\ref{fig:e7-page-a6}), which
carries out a twist of the Mahowald-Rezk calculation.

Finally, there is an Atkin-Lehner involution $w_{15}$ on ${\cal X}^D$,
whose effect is to take an abelian variety with $\Lambda$-action and
``twist'' the action by conjugating with $xy$. Since $-15$ has a
$2$-adic square root, the action of $w_{15}$ lifts to the
$2$-divisible group (see \ref{prop:atkinlehnerlift}) and so we get an
action of $\mb Z/2$ on $\TAF^D$.  The final goal of this paper is to
determine the homotopy groups of this homotopy fixed point spectrum,
which could be viewed as the global section object of an extension of
${\cal O}^{der}$ to the quotient stack ${\cal X}^D/w_{15}$.  This
calculation is, again, similar to the Mahowald-Rezk calculation, but
this time it adds an ideal carrying out two copies of the homotopy
fixed point spectral sequence for $KO$.

The exact connection of $\TMF_0(3)$ with $\TAF^D$ on the spectrum
level remains unclear.  Both map to a homotopy fixed point spectrum
$EO(G)$ associated to a finite subgroup $G$ of the extended Morava
stabilizer group, but this alone does not support the degree of
connection that is visible in homotopy theory.

\acks

The author would like to issue thanks to John Voight for assistance
with the fundamental domain, as well as general thanks towards Michael
Hill.  The author was partially supported by NSF grant 1206008 and a
fellowship from the Sloan foundation.

\section{Complex uniformization}

The Shimura curve $X^D$ of discriminant $15$ is of genus $1$, and it
has two elliptic points of order $3$.

In order to determine a complex uniformization of the curve, we first
fix an embedding $\tau$ of the division algebra $D$ into $M_2(\mb
R)$.  This differs from that
in~\cite[5.5]{alsina-bayer-shimuracurves}, and the resulting
fundamental domain differs from that in Figure 5.3 of loc. cit., which
was determined by Michon~\cite{michon-shimuracurves} and pictured
in~\cite[IV.\S 3.C]{vigneras-quaternions}.

Let $\tau\co D \to M_2(\mb R)$ be defined on the generators of
equation~(\ref{eq:divisiongens}) by
\begin{equation}
  \label{eq:embedding}
x \mapsto
\begin{bmatrix}
0 & \sqrt 3 \\
-\sqrt{3} & 0
\end{bmatrix},\ 
y \mapsto
\begin{bmatrix}
0 & \sqrt{5} \\
\sqrt{5} & 0
\end{bmatrix},\ 
xy \mapsto
\begin{bmatrix}
\sqrt{15} & 0\\
0 & -\sqrt{15}
\end{bmatrix}.
  \end{equation}
This determines the action of $\Lambda^\times$ on $\mb C \setminus \mb
R$, with quotient the complex Shimura curve $X^D_{\mb C}$.  Write
${\cal H}$ for the upper half-plane.

We have norm-1 elements in $\Lambda$:
\begin{align}
  \label{eq:generators}
\omega &= \frac{-1 + x}{2}, &
h &= 4 + xy, &
\gamma &= 4 + 5\omega^2 - 2y
\end{align}
The first rotates by $\frac{2\pi}{3}$ around an
elliptic point $i \in {\cal H}$, while the element $4 + xy$ is a
``principal homothety'' $h\co z \mapsto (4 + \sqrt{15})^2 z$.

We also define elements
\begin{equation}
  \label{eq:atkinleher}
 \tilde w_3 = x,\ \tilde w_5 = 5 + 2y,\ \tilde w_{15} = 5x + 2xy.
\end{equation}
These elements have norms $3$, $5$, and $15$ respectively.  They thus
act on the upper half-plane quotient as lifts of the Atkin-Lehner
involutions $w_3$, $w_5$, and $w_{15}$ respectively.  The element
$\tilde w_5$ is a hyperbolic translation stabilizing the circle of
radius $1$, while $\tilde w_3$ is an involution about this circle that
fixes $i$.

We note that in $\Lambda$, there is an identity
\[
5h = (\omega^2 \tilde w_5 \omega^2 + 5)y.
\]
Together with the fact that $y$ commutes with $\tilde w_5$ and
conjugate-commutes with $\omega$, this shows
\begin{equation}
  \label{eq:commuting}
(\omega^2 \tilde w_5 \omega^2) h = h (\omega \tilde w_5 \omega).
\end{equation}

\begin{prop}
There exists a hyperbolic hexagon (Figure~\ref{fig:fundamentaldomain})
which is a fundamental domain for the action of $\Lambda^{N=1}/\{\pm
1\}$ on ${\cal H}$.  It has the following six elliptic points as
vertices:
\begin{align}
  \label{eq:ellipticvertices}
v_1 &= h \omega \tilde w_5 i&
v_2 &= h i&
v_3 &= h \omega^2 \tilde w_5^{-1} i\notag\\
v_6 &= \omega^2 \tilde w_5 i&
v_5 &= i&
v_4 &= \omega \tilde w_5^{-1} i
\end{align}
This domain is symmetric about the imaginary axis, and the edges are
identified via
\begin{align*}
h\omega^2(\overrightarrow{v_5 v_6}) &= \overrightarrow{v_2 v_1},\\
h\omega(\overrightarrow{v_5 v_4}) &= \overrightarrow{v_2 v_3},\\
\gamma(\overrightarrow{v_4 v_3}) &= \overrightarrow{v_6 v_1}.
\end{align*}
\end{prop}
\begin{figure}[h]
\centerline{\includegraphics[width=15pc]{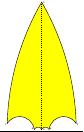}}
\caption{Fundamental domain for ${\cal X}^D_{\mb C}$}
\label{fig:fundamentaldomain}
\end{figure}

\begin{proof}
As elliptic points are preserved by the action of $\Lambda$, the
given vertices are all elliptic points in ${\cal H}$.  We note that
the operator $\omega^2 \tilde w_5 \omega^2$ takes $v_4$ to $v_5$ and
$v_5$ to $v_6$; equation~(\ref{eq:commuting}) implies that it also
takes $v_3$ to $v_2$ and $v_2$ to $v_1$.

We will now show that $v_1, \ldots, v_6$ are the vertices, indexed in
clockwise order, of a fundamental domain for the action of
$\Lambda^{N=1}$.  The resulting curve is of genus $1$.

The geodesics $\omega(\overrightarrow{v_5 v_6})$ and
$\omega^2(\overrightarrow{v_5 v_4})$ are along the unit circle.
Therefore, the original geodesics $\overrightarrow{v_5 v_6}$ and
$\overrightarrow{v_5 v_4}$ make angles of $\frac{\pi}{6}$ with the
imaginary axis, represented by $\overrightarrow{v_5 v_2}$.  Similarly,
$\overrightarrow{v_2 v_1}$ and $\overrightarrow{v_2 v_3}$ also make
angles of $\frac{\pi}{6}$ with $\overrightarrow{v_5 v_2}$.  The
hyperbolic volume is therefore $8\pi/3$, which is the volume of
$X^D$.  Therefore, once we have shown that the geodesic edges are
identified, this must be a fundamental domain.

The following identifications of geodesics are verified by checking
that they have the correct effect on endpoints.
\begin{itemize}
\item The transformation $a = h \omega^2$ takes $\overrightarrow{v_5 v_6}$
  to $\overrightarrow{v_2 v_1}$.
\item The transformation $b = h \omega$ takes $\overrightarrow{v_5 v_4}$
  to $\overrightarrow{v_2 v_3}$.
\item The transformation $\gamma = 4 + 5\omega^2 - 2y$, which is also realized
  by the element
\[
(\omega^2 \tilde w_5 \omega^2)^2 = 5(4 + 5\omega^2 - 2y),
\]
  takes $\overrightarrow{v_4 v_3}$ to $\overrightarrow{v_6 v_1}$.\qedhere
\end{itemize}
\end{proof}

This fundamental domain lets us produce a presentation of the group
$\Lambda^{N=1}/\{\pm 1\}$, which is the fundamental group of the
associated orbifold.

\begin{cor}
\label{cor:presentation}
There is a presentation
\begin{equation}
\Lambda^{N=1}/\{\pm 1\} = \langle h, \gamma, \omega \mid \omega^3 = (\omega^2 h^{-1} \gamma h \omega^2 \gamma^{-1})^3 = 1 \rangle.
\end{equation}
\end{cor}

\begin{proof}
The fundamental domain gives us the generators $a = h \omega^2$,
$b =  h \omega$, and $\gamma$, and relations
\[
(b^{-1} a)^3 = (b^{-1} \gamma a \gamma^{-1})^3 = 1.
\]
By switching generators to $h$, $\omega$, and $\gamma$, we obtain the
desired presentation.
\end{proof}

\section{Complex multiplication}

In this section we review results of Elkies on the structure of the
Shimura curve $X^D_{\mb Q}$ and its quotients. All of the results in
this section, and in particular the complex multiplication (CM)
points, are taken directly from
\cite[Section~5.2]{elkies-computations}.

The quotient $(X^D_{\mb Q})^*$ by the full group of Atkin-Lehner
involutions is of genus zero, with one elliptic point $P_6$ of order
$6$ (CM discriminant $-3$) and three elliptic points $P_2$, $P'_2$,
and $P''_2$ of order two (CM discriminants $-12$, $-15$, and $-60$
respectively).

There is thus a unique isomorphism $t\co (X^D)^* \to \mb P^1$ so that
$t(P_6) = \infty$, $t(P_2) = 0$, and $t(P_2'') = 1$. Elkies showed
that this coordinate satisfies $t(P'_2) = 81$, and that (among others)
there is a further rational point with $t$-coordinate $-27$ and CM
discriminant $-7$ \cite[Table~6]{elkies-computations}.

We now examine the Atkin-Lehner involutions $w_3$, $w_5$, and
$w_{15}$.  We recall that a point of $X^D$ is fixed by $w_d$ if and
only if it parametrizes a surface with a $\Lambda$-linear endomorphism
$t$ satisfying $t^2 + mt + d = 0$; this means that it has complex
multiplication by the splitting field of this polynomial, and that
this splitting field also splits $D$.

The involution $w_5$ fixes only those points which correspond to
points with an action by $\mb Z[\sqrt{-5}]$. However, since $3$ splits
in $\mathbb{Q}(\sqrt{-5})$, this field does not split $D$ and so the
involution $w_5$ has no fixed points.

Similarly, the involution $w_3$ on $X^D$ fixes only those points
parametrizing objects with an action of $\mb Z[\sqrt{-3}]$. These are
precisely the CM points with discriminants $-3$ and $-12$, and hence
the preimages of $P_6$ and $P_2$.

Finally, the involution $w_{15}$ fixes two points with an action of
$\mb Z[\sqrt{-15}]$. These must have CM discriminant $-15$ or $-60$,
and so are the preimages of $P_2'$ and $P_2''$.

The curve $(X^D_{\mb Q})/w_3$ has exactly one preimage each of $P_2$
and $P_6$.  The two preimages of $P_2'$ are complex conjugate and
their coordinates must generate the Hilbert class field of
$\mb{Q}(\sqrt{-15})$ by work of Shimura; this class field is $\mb
Q(\sqrt{-15},\sqrt{-3})$.  The preimages are thus defined over $\mb
Q(\sqrt{-3})$ and we can define a unique Galois-invariant coordinate
$s$ such that $s(P_6) = \infty$, $s(P_2) = 0$, and the preimages of
$P_2'$ are taken to $\pm \sqrt{-3}$.  This coordinate satisfies $t =
-3s^2$.

Then \cite[(76)]{elkies-computations} shows that there is a coordinate
$y$ on $X^D/w_{15}$ so that
\begin{equation}
  \label{eq:elkiesequation}
y^2 = -(3s^2 + 1)(s^2 + 27) = -(1-t)(27-t/3).
\end{equation}
Together $s$ and $y$ generate the function field of the curve.  By
construction, $y$ is fixed by $w_{15}$ and negated by $w_3$ and $w_5$,
while $s$ is fixed by $w_3$ and negated by $w_5$ and $w_{15}$.

\section{Rational uniformization}

In this section, we review structure of the rational Shimura curve
$X^D_{\mb Q}$ and its quotients, continuing notation from the previous
section.

\begin{prop}
The curves $(X^D_{\mb Q})/w_3$ and $(X^D_{\mb Q})/w_{15}$ are both
isomorphic to $\mb P^1$, and the natural map from $X^D$ into their
product is an embedding.
\end{prop}

\begin{proof}
The map $X^D_{\mb Q} \to (X^D_{\mb Q})/w_3 \times (X^D_{\mb
Q})/w_{15}$ only fails to be at an embedding at places fixed by
both $w_3$ and $w_{15}$, hence $w_5$.

The involution $w_5$, with no fixed points, acts as an order-2
translation on the genus $1$ curve. The quotient $X^D/w_5$ is of genus
one, and under the projection $X^D/w_5 \to (X^D)^*$ there is a single
preimage of $P_6$ which is an elliptic point of order $3$. (The
subspace of Figure~\ref{fig:fundamentaldomain} with positive real part
is a fundamental domain for $X^D_{\mb C}/w_5$.)

The quotient curve $X^D/w_3$ is of genus zero with two elliptic points
of order $6$ (the preimages of $P_6$) and two of order $2$ (the
preimages of $P_2$).  As $X^D/w_3$ is smooth of genus zero and has
points defined over $\mb Q$, it is isomorphic to $\mb P^1$ over $\mb
Z[1/15]$.

The quotient curve $X^D/w_{15}$ is also of genus zero, with one
elliptic point of order 3 (the preimage of $P_6$) and four of order 2
(the preimages of $P_2'$ and $P_2''$).  Similarly, it is isomorphic to
$\mb P^1$ over $\mb Z[1/15]$.
\end{proof}

\begin{prop}
\label{prop:definingequation}
There are meromorphic functions
\begin{align*}
u\co (X^D_{\mb Q})/w_3 &\overto^\sim \mb P^1_{\mb Q} \text{ and}\\
v\co (X^D_{\mb Q})/w_{15} &\overto^\sim \mb P^1_{\mb Q}  
\end{align*}
so that the resulting embedding $X^D_{\mb Q} \to \mb P^1_{\mb Q}
\times \mb P^1_{\mb Q}$ has, as image, the (closure of the) set of
solutions of
\begin{equation}
  \label{eq:defining}
(u^2 + u + 1)(v^2 + v + 1) = \tfrac{5}{9}.
\end{equation}
The Atkin-Lehner operator $w_3$ fixes $u$ and sends $v$ to $-1-v$,
while $w_{15}$ fixes $v$ and sends $u$ to $-1-u$.
\end{prop}

\begin{proof}
We will first change the coordinates $t$, $s$, and $y$, from the
previous section, into new ones that will ultimately prove
better-behaved integrally.

Let $w = \frac{t-1}{4}$ denote the ``equivalent, $2$-adically good''
coordinate Elkies describes on $(X^D_{\mb Q})^*$. We also let $u =
\frac{s-3}{6}$ and $v = \frac{3y - (t-81)}{2(t-81)}$ so that $t =
4w+1$, $s = 6u+3$, and $y = \frac{4(2v+1)(w-20)}{3}$.  These still
generate the function field.

The equation $t = -3s^2$ is equivalent to $w = -27u^2 -27u - 7$.

The equation $y^2 = -(1-t)(27-t/3)$ becomes $v^2 + v =
\frac{w-5}{20-w}$, or $w = 20 - \frac{15}{v^2 + v + 1}$.

Putting these two equations for $w$ together, we find that $u$ and $v$
generate the function field of the curve $X^D_{\mb Q}$ and satisfy the
formula of equation~(\ref{eq:defining}).

Together the coordinates $u$ and $v$ determine an embedding $X^D_{\mb
  Q} \to \mb P^1_{\mb Q} \times \mb P^1_{\mb Q}$ over $\mb Q$,
representing the map $X^D_{\mb Q} \to X^D_{\mb Q}/w_3 \times X^D_{\mb
  Q}/w_{15}$.
\end{proof}

\section{CM points and integral uniformation}

We recall the intersection theory employed in \cite{shimc}, based on
\cite{kudla-rapoport-yang}.  Given an element $z \in \mb C$ generating
a quadratic imaginary field, let $D_z$ be the divisor on $X^D_{\mb
  Z[1/15]}$ parametrizing points with complex multiplication by
$z$. The following is derived by knowing that when two such divisors
with distinct associated fields intersect, the resulting point has two
distinct types of complex multiplication and thus must represent a
supersingular point.

\begin{prop}
Let two orders be generated by elements $x$ and $y$ contained in
nonisomorphic quadratic imaginary fields, with $d_x$ and $d_y$ their
respective discriminants.  If $D_x \cap D_y$ contains points in
characteristic $p$, there must be an integer $m$ such that the
quantity
\begin{equation}
  \label{eq:disc}
\Delta = (2m + \Tr(x)\Tr(y))^2 - \Nm(x) \Nm(y)
\end{equation}
produces Hilbert symbols
\begin{equation}
  \label{eq:hilbert}
(\Delta, d_y)_q = (\Delta,d_x)_q
\end{equation}
which are nontrivial precisely when $q \in \{p,\infty,3,5\}$.  (In
particular, $\Delta$ must be negative.)
\end{prop}

We now choose a point $Q$ on $X^D/w_3$ with $s$-coordinate
$3$. This makes $t = -3s^2$ equal $-27$, which was already established
to give it CM discriminant $-7$.

Table~\ref{tab:cm-points} collects together the values of the
coordinates $u$ and $v$ on the (preimages of) complex multiplication
points discussed in the previous sections.  Where not explicitly
stated previously, these are explicitly derived from the equations
relating these six coordinates. By convention, in the following table
$\omega$ is the third root of unity $\frac{-1 + \sqrt{-3}}{2}$.

\begin{table}[h]
  \centering
\begin{tabular}{cc|ccc|ccc}
Point&CM disc&$t$&$s$&$y$&$w$&$u$&$v$\\
\hline
$P_6$&$-3$&
$\infty$&$\infty$&$\infty$&
$\infty$&$\infty$&$\omega$\\
$P_2$&$-12$&
$0$&$0$&$3 \sqrt{-3}$&
$\frac{-1}{4}$&$\frac{-1}{2}$&$\frac{-3 + 1/\sqrt{-3}}{6}$\\
$P'_2$&$-15$&
$81$&$3\sqrt{-3}$&$0$&
$20$&$\omega$&$\infty$\\
$P''_2$&$-60$&
$1$&$1/\sqrt{-3}$&$0$&
$0$&$\frac{\omega-4}{9}$&$\frac{-1}{2}$\\
$Q$&$-7$&
$-27$&$3$&$12\sqrt{-7}$&
$-7$&$0$&$\frac{\sqrt{-7} - 3}{6}$\\
\end{tabular}
\caption{Complex multiplication points}
\label{tab:cm-points}
\end{table}

We thus consider the divisors on $X_{\mb Z[1/15]}$ with complex
multiplication by $\frac{-1 + \sqrt{-3}}{2}$, by $\frac{-1 +
  \sqrt{-7}}{2}$, and by $\frac{-1 + \sqrt{-15}}{2}$.  These have
discriminants $-3$, $-7$, and $-15$ respectively, and their divisors
are associated to the primages of the points $P_6$, $P_2'$, and $Q$ in
table~\ref{tab:cm-points}.

\begin{prop}
The divisors associated to $P_6$, $P_2'$, and $Q$ do not intersect on $X^D$.
\end{prop}

\begin{proof}
Table~\ref{tab:cm-intersections} summarizes the possible valid
discriminants $\Delta$ from equation~(\ref{eq:disc}) and the list of
primes with nonvanishing Hilbert symbols from
equation~(\ref{eq:hilbert}) for all possible pairs.
\begin{table}[h]
  \centering
  \begin{tabular}{cc|cc|l}
    Points&\phantom{Points}&$m$&$\Delta$&Nonvanishing Hilbert symbols\\
\hline
$P_6$ & Q
 &  $0,-1$ & $-20$ & $\{5,\infty\}$\\
&&  $1,-2$ & $-12$ & $\{3,\infty\}$\\
\hline
$P_6$ & $P_2'$
   &  $0,-1$ & $-44$ & $\{11,\infty\}$\\
  &&  $1,-2$ & $-36$ & $\{3,\infty\}$\\
  &&  $2,-3$ & $-20$ & $\{5,\infty\}$\\
\hline
$Q$ & $P_2'$
   &  $0,-1$ & $-104$ & $\{13,\infty\}$\\ 
  &&  $1,-2$ & $-96$ & $\{3,\infty\}$\\
  &&  $2,-3$ & $-80$ & $\{5,\infty\}$\\
  &&  $3,-4$ & $-56$ & $\{7,\infty\}$\\
  &&  $4,-5$ & $-24$ & $\{3,\infty\}$
  \end{tabular}
  \caption{Complex multiplication intersections}
  \label{tab:cm-intersections}
\end{table}

As a result, none of these three divisors intersect at any prime in
$\mb Z[1/15]$.
\end{proof}

\begin{prop}
The embedding $u \times v$ of Proposition~\ref{prop:definingequation}
extends to an embedding $X^D \to \mb P^1 \times \mb P^1$ over $\mb
Z[1/15]$, giving $X^D$ the same defining equation integrally.
\end{prop}

\begin{proof}
Using data in table~\ref{tab:cm-points}, we find that the coordinate
$u$ always take distinct values on $P_6$, $P'_2$, and $Q$ over $\mb
Z[1/15]$.  As a result, the isomorphism $u$, from $X^D_{\mb Q}/w_3$ to
$\mb P^1_{\mb Q}$, extends to an isomorphism from $X^D_{\mb Z[1/15]}/w_3
\to \mb P^1_{\mb Z[1/15]}$.

The element $\frac{\sqrt{-7} - 3}{6}$ has minimal polynomial $x^2 + x
+ \frac{4}{9}$, and hence is an algebraic integer over $\mb Z[1/15]$
(so it never coincides with $\infty$ in $\mb P^1$ over this ring).
Moreover, substituting $u=0$ into equation~(\ref{eq:defining}) we also
find that
\[
\left(\tfrac{\sqrt{-7} - 3}{6} - \omega\right)\left(\tfrac{\sqrt{-7} - 3}{6} - \omega^2\right) =
\tfrac{5}{9},
\]
so the element $\frac{\sqrt{-7}-3}{6} - \omega$ is a unit over $\mb
Z[1/15]$.  This shows that $v$ also takes distinct values on $P_6$,
$P'_2$, and $Q$ over $\mb Z[1/15]$.  Therefore, the map $v$ from
$X^D_{\mb Q}/w_{15}$ to $\mb P^1_{\mb Q}$ extends to an isomorphism
from $X^D_{\mb Z[1/15]}/w_{15} \to \mb P^1_{\mb Z[1/15]}$.

This demonstrates that equation~(\ref{eq:defining}) gives a valid
description, over $\mb Z[1/15]$, of $X^D$ as a closed
subscheme of $\mb P^1 \times \mb P^1$.
\end{proof}

\section{Sections of the cotangent bundle}

Let $\kappa_{\cal X}$ be the relative cotangent bundle of ${\cal
  X}^D$, and $\kappa_X$ the (trivial) relative cotangent bundle of the
coarse moduli scheme $X^D$.

As the map ${\cal X}^D \to X^D$ is triply ramified precisely over the
degree-two divisor $P_6$, there is an isomorphism
\begin{equation}
  \label{eq:contangentautomorphic}
H^0({\cal X}^D; \kappa_{\cal X}^{\otimes t}) \cong H^0(X^D; {\cal O}(\lfloor
2t/3 \rfloor P_6)).
\end{equation}

\begin{prop}
There is a section $a_2 \in H^0({\cal X}^D; \kappa_{\cal X})$ which
has double zeros on the divisor $P_6$ and nowhere else.
\end{prop}

\begin{proof}
There is a nowhere-vanishing 1-form on $X^D$, given by
\[
\frac{du}{(u^2+u+1)(2v+1)} = \frac{-dv}{(v^2+v+1)(2u+1)}.
\]
This exhibits $\kappa_X$ as trivial.  We denote the pullback of this
to ${\cal X}^D$ by $a_2$.  This section $a_2$ has a double pole
along the divisor defining $P_6$ (which should be regarded as a $2/3$
pole).  (Note that multiplication by $a_2^n$ induces the isomorphism
(\ref{eq:contangentautomorphic}).)
\end{proof}

As in \cite[2.16]{shimc}, there is an identification of $\kappa_{\cal
X}$ with the exterior square $\bigwedge^2 \Omega$ of the relative
cotangent bundle at the zero section of the universal abelian scheme
${\cal A} \to {\cal X}^D$.

We will choose an isomorphism $\Lambda \otimes \mb Z_2 \cong M_2(\mb
Z_2)$.  This isomorpism determines, on ${\cal X}^D_{\mb Z_2}$, a
splitting $\Omega \cong \omega \oplus \omega$, and thus an isomorphism
of $\kappa_{\cal X}$ with $\omega^2$.

\begin{cor}
\label{cor:squareobstruction}
There is a 2-torsion element in the Picard group of ${\cal X}^D_{\mb
  Z_2}$ which defines the obstruction to expressing $a_2$ as a unit
times the square of a section $a_1$ of $\omega$.
\end{cor}

\begin{proof}
The section $a_2$, with double zeros along the divisor $P_6$,
defines an isomorphism of line bundles ${\cal O}(-2P_6) \cong
\kappa_{\cal X}$, or equivalently an isomorphism ${\cal O} \cong
\kappa_X(2P_6)$.  The two-torsion line bundle $\omega(P_6)$ then
provides precisely this obstruction.
\end{proof}

\section{Double covers}

In this section we will analyze double covers on ${\cal X}^D$.  We
begin with the following.
\begin{prop}
\label{prop:alglevelstruct}
For $p \in \{2,3,5\}$, there exists a nontrivial homomorphism
$\sigma_p\co \Lambda^{N=1} \to \mb Z/2$ obtained by imposing level
structure at the prime $p$.  The effect on the generators of
Corollary~\ref{cor:presentation} is given in
Table~\ref{tab:level-struct}.
\begin{table}[h]
  \centering
\begin{tabular}{c|ccc}
Generator&$\sigma_2$&$\sigma_3$&$\sigma_5$\\
\hline
$h$&      $1$& $0$& $1$\\
$\gamma$& $0$& $1$& $1$\\
$-1$&     $0$& $0$& $1$\\
\end{tabular}
\caption{Images of generators under double covers}
\label{tab:level-struct}
\end{table}
\end{prop}

\begin{proof}
We first consider the ring map $\Lambda \to M_2(\mb Z/2)$, given by
\[
\omega \mapsto
\begin{bmatrix}
0 & 1 \\
1 & 1
\end{bmatrix},\ 
y \mapsto
\begin{bmatrix}
0 & 1 \\
1 & 0
\end{bmatrix}.
\]
In particular, $x = 2\omega + 1$ maps to $1$.  Under this map,
\[
h \mapsto
\begin{bmatrix}
0 & 1 \\
1 & 0
\end{bmatrix},\ 
\gamma \mapsto
\begin{bmatrix}
1 & 1 \\
1 & 0
\end{bmatrix},\ 
-1 \mapsto
\begin{bmatrix}
1 & 0 \\
0 & 1
\end{bmatrix}.\ 
\]
The composite homomorphism
\[
\sigma_2\co \Lambda^{N=1} \to \GL_2(\mb Z/2) \to \{\pm 1\},
\]
which sends a matrix to the sign of its permutation action on
$(\mb Z/2)^2$, is then trivial on $\gamma$ and $-1$ and nontrivial on
$h$.

We next have a ring homomorphism $\Lambda \to \mb F_9$ sending
$\omega$ to $1$ and $y$ to an element which is a square root of $-1$,
generating $\mb F_9$ over $\mb F_3$.  Under this map,
\[
h \mapsto 1,\ 
\gamma \mapsto \sqrt{-1},\ 
-1 \mapsto -1.
\]
The composite map
\[
\sigma_3\co \Lambda^{N=1} \to (\mb F_9)^{N=1} \cong \mb Z/4
\to \mb Z/2
\]
is trivial on $h$ and $-1$, and nontrivial on $\gamma$.

Finally, we have a ring homomorphism $\Lambda \to \mb F_{25}$, sending
$\omega$ to a third root of unity and $y$ to zero.  Under this map,
\[
h \mapsto -1,\ 
\gamma \mapsto -1,\ 
-1 \mapsto -1.
\]
The composite map
\[
\sigma_5\co \Lambda^{N=1} \to (\mb F_{25})^{N=1} \cong \mb Z/6 \to \mb Z/2
\]
is nontrivial on all three generators.
\end{proof}

This shows that all elements of $H^1({\cal X}^D_{\mb C};\mb Z/2)$ come
from some combination of these three algebraically defined level
structures.

\begin{prop}
\label{prop:complexdoublecover}
The character $\sigma_3 \sigma_5$ induces a double cover ${\cal Y_{\mb
    C}} \to {\cal X}^D_{\mb C}$, obtained by imposing level structure
away from the prime $2$, such that on ${\cal Y}_{\mb C}$ the form
$a_2$ has a square root.
\end{prop}

\begin{proof}
Using the embedding $\mb Z_2 \to \mb C$, the section $a_2$ of
$\omega^2$ is an automorphic form of weight $2$.  We represent $a_2$
by a holomorphic function $a_2(z)$ on the upper half-plane ${\cal
H}$, with double zeros at the elliptic points.  It has a square
root $a_1(z)$, which is almost an automorphic form of weight $1$.  The
group $\Lambda^{N=1}$ acts on $a_1$ by a character $\Lambda^{N=1} \to
\{\pm 1\}$.

The {\em real} structure $z \mapsto -\bar z$ is compatible with the
action of $\Lambda^{N=1}$, and so every automorphic form $f(z)$ has a
conjugate $\bar f(z) = \overline {f(-\bar z)}$.  As $a_1(z)$ and $\bar
a_1(z)$ have the same zeros (concentrated at the elliptic points),
they differ by a complex scalar; by rescaling we may assume that
$a_1(z)$ takes real values on the imaginary axis.

We note that $(-1)$ must act nontrivially on $a_1$, as it acts by
negation on all forms of weight $1$.  Therefore,
Table~\ref{tab:level-struct} tells us that the character of
$\Lambda^{N=1}$ acting on $a_1$ is in the set $\{\sigma_5, \sigma_2
\sigma_5, \sigma_3 \sigma_5, \sigma_2 \sigma_3 \sigma_5\}$.

The element $h$ sends $a_1(z)$ to $a_1((4 + \sqrt{15})^2 z) /
(4-\sqrt{15})$, which must be $\pm a_1(z)$.  The function $t \mapsto
a_1(it)$ is an analytic function on the positive real line, taking
real values, with zeros only at the points $(4+\sqrt{15})^{2k}i$.  In
addition, these are simple zeros, which forces $a_1(it)$ to change
signs.  Thus $h(a_1(z)) = -a_1(z)$.

By Table~\ref{tab:level-struct}, this further reduces the
possibilities for the character of $\Lambda^{N=1}$ acting on $a_1$ to
the set $\{\sigma_3 \sigma_5, \sigma_5\}$.

The Eichler-Selberg trace formula \cite{miyake} shows that the ring
endomorphism associated to the Atkin-Lehner involution $w_5$ has trace
$5$ on the collection of forms of weight $2$, which is one-dimensional
and generated by $a_2(z)$.  Therefore, $w_5$ must act on $a_1(z)$ by
sending it to $\pm \sqrt{5} a_1(z)$.  However, we have the identity $5
\gamma = \tilde w_5^2$ in $\Lambda$.  This shows that $\gamma$ acts on
$a_1(z)$ trivially, fixing the character as $\sigma_3 \sigma_5$.
This character can thus be obtained by imposing level structure away
from the prime $2$.
\end{proof}

\begin{defn}
Let $K$ denote the nontrivial unramified extension of $\mb Q_2$ of
degree $4$, with ring of integers ${\cal O}_K = W(\mb F_{16})$.
\end{defn}

\begin{prop}
After choosing a complex embedding $K \to \mb C$, there is an
\'etale double cover ${\cal Y} \to {\cal X}^D_{{\cal O}_K}$ over
$\Spec({\cal O}_K)$ inducing the double cover ${\cal Y}_{\mb C} \to
{\cal X}^D_{\mb C}$. The composite ${\cal Y} \to {\cal X}^D_{\mb Z_2}$
is a Galois cover with Galois group $\mb Z/8$.
\end{prop}

\begin{proof}
The \'etale fundamental groups fit into the following diagram of
groups with exact rows:
\[
\xymatrix{
1 \ar[r] &
\comp{\pi_1({\cal X}^D_\mb C)} \ar[r] \ar@{>>}[d] &
\pi_1^{et}({\cal X}^D_{\mb Z_2}) \ar[r] \ar[d] &
\pi_1^{et}(\mb Z_2) \ar[r] \ar[d]^{\chi} &
1 \\
1 \ar[r] &
(\Lambda/(xy))^{N=1} \ar[r] &
(\Lambda/(xy))^\times \ar[r] &
(\mb Z/15)^\times \ar[r] &
1
}
\]
Here $\chi$ is the cyclotomic character, and the image of $\chi$ is
generated by the element $2$---the image of the Frobenius
automorphism. The images of the top row in the bottom row 
then fit into a commutative diagram as follows.
\[
\xymatrix{
1 \ar[r] &
(\mb F_9)^{N=1} \times (\mb F_{25})^{N=1} \ar[r] \ar@{>>}[d]^{\sigma_3 \sigma_5} &
Im(\pi_1^{et}({\cal X}^D_{\mb Z_2})) \ar[r] \ar@{>>}[d] &
\langle 2 \rangle \ar[r] \ar[d]^\sim & 1\\
1 \ar[r] &
\mb Z/2 \ar[r] &
G \ar[r] &
\mb Z/4 \ar[r] &
1.
}
\]
This explicitly determines an index-$8$ subgroup of $\pi_1^{et}({\cal
  X}^D_{\mb Z_2})$ with quotient $G$, determining a composite
degree-$8$ cover
\[
{\cal Y} \to {\cal X}^D_{{\cal O}_K} \to {\cal X}^D_{\mb Z_2}.
\]
Moreover, the pullback of ${\cal Y}$ along ${\cal X}^D_{\mb C} \to
{\cal X}^D_{{\cal O}_K}$ is the double cover determined by the character
$\sigma_3 \sigma_5$, which is therefore ${\cal Y}_{\mb C}$.

We then verify directly (using the fact that $\mb F_9^\times$ and $\mb
F_{25}^\times$ are cyclic) that the Frobenius element lifts to a
generator of $G$, so that $G \cong \mb Z/8$.
\end{proof}

\begin{rmk}
As the element $-1$ has nontrivial image in $\{\pm 1\}$ under the
character $\sigma_3 \sigma_5$, the double cover ${\cal Y}$ has the
same geometric points as ${\cal X}^D_{\mb Z_2}$ and ${\cal X}^D_{{\cal
    O}_K}$.
\end{rmk}

\begin{rmk}
\label{rmk:torsionorbits}
The recipe in this proof allows us to describe the stack ${\cal Y}$ as
follows. There is a degree-192 cover ${\cal X}^D_{\mb Z_2}(\sqrt{15})
\to {\cal X}^D_{\mb Z_2}$ parametrizing $2$-dimensional abelian
schemes $A$ with $\Lambda$-action, together with a choice of primitive
torsion point for $xy \in \Lambda$; the group $(\Lambda/(xy))^\times$
is a Galois group for this cover, acting by the $\Lambda$-action on
the torsion point. There are two components for the cover, each
stabilized by $Im(\pi_1^{et}({\cal X}^D_{\mb Z_2}))$. The stabilizer
group of one such component then contains an index-eight subgroup $H
\subset (\Lambda/(xy))^\times$, and ${\cal Y}$ is the stack-theoretic
quotient of this component by $H$.
\end{rmk}
\begin{prop}
\label{prop:integraldoublecover}
On the double cover ${\cal Y} \to {\cal X}^D_{{\cal O}_K}$ of
Proposition~\ref{prop:complexdoublecover}, an ${\cal O}_K$-unit times
the restriction of $a_2$ has a square root.
\end{prop}

\begin{proof}
The \'etale fundamental group of ${\cal Y}_{K}$ fits into an exact sequence
\[
1 \to \comp{\pi_1({\cal Y}_{\mb C})} \to \pi_1^{et}({\cal Y}_{K})
\to Gal(\overline{\mb Q}_2/K) \to 1,
\]
which gives an exact sequence in cohomology
\begin{equation}
  \label{eq:infres}
0 \to K^\times/(K^\times)^2 \to H^1({\cal Y}_{K};\mb Z/2)\to
H^1({\cal Y}_{\mb C}; \mb Z/2).
\end{equation}
There is an isomorphism $\mb Z/2 \to \mu_2$ of \'etale sheaves on
${\cal Y}_{K}$, and so the \'etale cohomology group
$\Hom(\pi_1^{et}({\cal Y}_{L}), \mb Z/2)$ is the same as the
\'etale cohomology group $H^1({\cal Y}_{K}; \mu_2)$
classifying $2$-torsion line bundles.

The 2-torsion line bundle $\omega(P_6)$ of
Corollary~\ref{cor:squareobstruction}, with chosen trivialization
$a_2$ of its square, is classified by an element $\theta \in H^1({\cal
  X}^D_{K};\mu_2)$. By
Proposition~\ref{prop:complexdoublecover}, the image of $\theta$ in
$H^1({\cal Y}_{\mb C}; \mb Z/2)$ vanishes.  The exact
sequence~(\ref{eq:infres}) then shows that there is a form $a_1$ of
weight one on ${\cal Y}_{K}$ such that $a_1^2 = u a_2$ for a
$K$-unit $u$.

Twice the divisor associated to $a_1$ is the pullback to ${\cal Y}$ of
the divisor associated to $u a_2$.  As ${\cal Y} \to {\cal X}^D_{\mb
  Z_2}$ is unramified over $2$, $u a_2$ can only have an even-degree
zero or pole over the divisor $(2)$ of ${\cal Y}$, and no other poles.
We can therefore rescale $a_1$ by a multiple of $2$ to lift it to a
form defined on ${\cal Y}$ whose square is a ${\cal O}_K$-unit times
$a_2$.
\end{proof}

\section{Automorphic forms on the cover}

Let ${\cal Y}$ be the double cover of ${\cal X}^D_{{\cal O}_K}$ from the
previous section, with section $a_1 \in H^0({\cal Y}; \omega)$ with
simple zeros at the points lying over $P_6$.  The coarse moduli scheme
of ${\cal Y}$ is $X^D_{{\cal O}_K}$.

On ${\cal Y}$, the map taking a section $s$ to $s/a_1^t$ gives an
identification of sections of $\omega^{\otimes t}$ with elements in
the fraction field of $X^D_{{\cal O}_K}$ with poles of degree at most
$\lfloor 2t/3\rfloor$ along the divisor $\{u = \infty\}$ defining
$P_6$.  In this section, we will determine this graded ring.

\begin{prop}
The graded ring
\begin{equation}
  \label{eq:automorphicforms}
\bigoplus_{t=0}^\infty H^0(X^D_{{\cal O}_K}; {\cal O}(\lfloor 2t/3 \rfloor P_6))
\end{equation}
is isomorphic to a graded ring
\[
{\cal O}_K[a_1, a_3, a_6] / f(a_i)
\]
Here $a_1$ has degree $1$, $a_3$ has degree $3$, $a_6$ has degree $6$,
and
\[
f(a_i) = a_6^2 + a_6(a_1^6 + a_1^3 a_3 + a_3^2) + (a_1^6 + a_1^3 a_3 + a_3^2)^2 - \tfrac{5}{9}a_1^6 (a_1^6 + a_1^3 a_3 + a_3^2).
\]
\end{prop}

\begin{proof}
From equation~(\ref{eq:defining}), the fraction field of $X^D_K$
is
\begin{equation}
  \label{eq:fractionfield}
K(X^D) = K(u)[v] / (u^2+u+1)(v^2+v+1) - \tfrac{5}{9}.
\end{equation}
A generic element of this field is $a(u) + b(u) v$.  We find that this
has no poles in the coordinate chart $u,v \neq \infty$ only when
$a(u)$ and $b(u)$ are in ${\cal O}_K[u]$, and then has no poles in
the coordinate chart $u \neq \infty, v \neq 0$ only when $(1 + u +
u^2)$ divides $b(u)$.  Finally, the degree of the pole of $a(u) +
b'(u) (1+u+u^2)v$ over $P_6$ is the total degree in $u$.

As a consequence, the ring of equation~(\ref{eq:automorphicforms}) is
the graded ring generated by three elements, of weights $1$, $3$, and
$6$:
\begin{align*}
a_1 & & a_3 &= u a_1^3 &a_6 &= (1+u+u^2)va_1^6
\end{align*}

The defining relation satisfied by $u$ and $v$, after multiplying by
$a_1^{12} (1 + u + u^2)$, becomes the minimial polynomial for $a_6$:
\[
a_6^2 + a_6(a_1^6 + a_1^3 a_3 + a_3^2) + (a_1^6 + a_1^3 a_3 + a_3^2)^2
= \tfrac{5}{9}a_1^6 (a_1^6 + a_1^3 a_3 + a_3^2)
\]
This gives a presentation of the ring of automorphic forms as desired.
\end{proof}

\begin{cor}
\label{cor:coefficientring}
The $\mb Z/8$-action on the graded ring of automorphic forms
$H^0({\cal Y}; \omega^{\otimes t})$ is determined as follows. On the
scalars ${\cal O}_K$ this action factors through the quotient $\mb Z/8
\twoheadrightarrow Gal(K/\mb Q_2)$, and the generator of $\Bbb Z/8$
sends $a_k$ to $y^k \cdot a_k$ for an element $y \in {\cal
  O}_K$ with norm $-1$ which is congruent to $1$ mod $2$.
\end{cor}

\begin{proof}
Since the group $H^0({\cal Y}, \omega)$ is isomorphic to ${\cal O}_K$
with generator $a_1$, the generator $\sigma$ of $\Bbb Z/8$ sends $a_1$
to $y \cdot a_1$ for some element $y \in {\cal O}_K^\times$. The element
$\sigma^4$ acts on $a_1$ by negation, which is equivalent to
$N_{K/\mb Q}(y) = -1$.

In $\mb F_{16} = {\cal O}_K/(2)$, Hilbert's theorem 90 implies that $y
\equiv \sigma w / w$ for some unit $w$. By rescaling forms in degree
$k$ by $w^k$, without loss of generality we may assume that $y$
reduces to $1$ in $\mb F_{16}$.

As $a_3 = u a_1^3$ and $a_6 = (1 + u + u^2) v a_1^6$ were obtained by
multiplying $a_1^k$ by $Gal(K/\mb Q_2)$-invariant rational functions on
$X^D_{\mb Q_2}$, we obtain the desired action on the remainder of the
ring.
\end{proof}

\section{Cohomology of ${\cal X}^D_{\mb Z_2}$}

In this section we will exploit the ring of automorphic forms from
Corollary~(\ref{cor:coefficientring}) to determine the cohomology of
${\cal X}^D_{\mb Z_2}$ with coefficients in the tensor powers of
$\omega$.

We first note that ${\cal Y}$ satisfies a type of Serre duality.
\begin{prop}
The cup product creates a perfect pairing
\[
H^s({\cal Y}; \omega^{\otimes t}) \otimes H^{1-s}({\cal Y};
\omega^{\otimes(2-t)}) \to H^1({\cal Y}; \omega^{\otimes 2}) \cong \mb
{\cal O}_K.
\]
\end{prop}

\begin{proof}
The stack ${\cal Y}$ has a cover by the open substacks $a_1^{-1}
{\cal Y}$ and $a_3^{-1} {\cal Y}$, each of which has an affine scheme
as its coarse moduli object and whose higher cohomology vanishes
because all points have automorphism groups of odd order. In terms of
the graded ring of Corollary~\ref{cor:coefficientring}, the
Mayer-Vietoris sequence then degenerates to an exact sequence
\[
0 \to R_* \to a_1^{-1} R_* \oplus a_3^{-1} R_* \to (a_1 a_3)^{-1}
R_* \to H^1({\cal Y}; \omega^{\otimes *}) \to 0.
\]
In particular, the element $D = (a_1 a_3)^{-1} a_6$ maps to a
generating element in $H^1$, which makes the monomial basis $a_1^k
a_3^l a_6^\epsilon$ of $H^0$ dual to the basis elements which are the
image of $(a_1)^{-1-k} (a_3)^{-1-l} (a_6)^{1-\epsilon}$.
\end{proof}

We now employ the cohomology spectral sequence
\[
H^p(\mb Z/8; H^q({\cal Y}; \omega^{\otimes t})) \Rightarrow
H^{p+q}({\cal X}^D_{\mb Z_2}; \omega^{\otimes t}).
\]
This spectral sequence is concentrated either in $q=0$ or $q=1$,
according to the value of $t$.  Therefore, this collapses to the group
cohomology, with terms having degrees shifted according to whether
they originate in $H^0({\cal Y})$ or $H^1({\cal Y})$.

\begin{thm}
Consider the bigraded ring
\[
H^s({\cal Y}; \omega^{\otimes t}) [\zeta] /
  (2\zeta),
\]
where $\zeta$ is in degree $(s,t) = (1,0)$.  This ring has an
additional $\mb Z/2$-grading: $a_1$, $a_3$, and $\zeta$ have odd
grading, while $a_6$ and $D$ have even grading.

The cohomology of $\mb Z/2 \subset \mb Z/8$ with with coefficients in
$H^s({\cal Y}; \omega^{\otimes t})$ is the subgroup of elements in
bidegree $(s,t)$ of even grading, and the cohomology of $\mb Z/8$ with
with coefficients in $H^s({\cal Y}; \omega^{\otimes t})$ consists of
those elements invariant under the map $x \mapsto y^t \sigma(x)$, where $y$ is
the element of norm $-1$ from Corollary~\ref{cor:coefficientring} and
$\sigma$ is a generator of the Galois group.
\end{thm}

\begin{proof}
The calculation of the cohomology of $\mb Z/2$ is standard. In
addition, because the extension $\mb Z_2 \to {\cal O}_K$ is unramified
the terms in the spectral sequence
\[
H^p(\mb Z/4; H^q(\mb Z/2; H^*({\cal Y}; \omega^{\otimes t})))
  \Rightarrow H^{p+q}(\mb Z/8; H^*({\cal Y}; \omega^{\otimes t}))
\]
vanish for $p > 0$. Therefore,
\[
H^q(\mb Z/8; H^*({\cal Y}; \omega^{\otimes t})) \cong 
H^q(\mb Z/2; H^*({\cal Y}; \omega^{\otimes t}))^{Gal(K/\mb Q_2)},
\]
as desired.
\end{proof}

As $y$ is congruent to $1$ mod $2$, elements involving a positive
power of $\zeta$ are invariant under the action of $Gal(K/\mb Q_2)$
if and only if their coefficients come from $\mb F_2 \subset \mb
F_{16}$.

As $y^2$ has norm $1$, Hilbert's theorem 90 implies that there exists
an element $x \in {\cal O}_K^\times$ with $y^2 = x / \sigma(x)$; the
element $x$ must also reduce to $1$ in $\mb F_{16}$. We formally
define $b_i = x^{i/2} a_i$.

Then the elements $b_i$ generate a ring $\mb
Z_2[b_1,b_3,b_6]/(f(b_i))$. The ring $H^0(\mb Z/8; H^0({\cal
  Y};\omega^{\otimes t}))$ is the subring consisting of those elements
in even total weight. Similarly, the module $H^0(\mb Z/8; H^1({\cal
  Y};\omega^{\otimes t}))$ is, as a consequence of the Mayer-Vietoris
sequence, generated by elements of the form $(b_1)^{-1-j} (b_3)^{-1-k}
b_6^{1-\epsilon}$ in even total weight.

\section{The Atkin-Lehner involution}

The following is an expansion on what appears in \cite[5.6]{shimc}.
\begin{prop}
\label{prop:atkinlehnerlift}
The $2$-divisible group on ${\cal X}^D$ descends to one on the
quotient stack ${\cal X}^D/w_{15}$.
\end{prop}

\begin{proof}
We first recall the definition of the quotient stack.

For a scheme $Y$, the $Y$-points ${\cal X}^D(Y)$ form the category of
$2$-dimensional abelian schemes $A$ over $Y$ with an action of
$\Lambda$, and morphisms are $\Lambda$-linear isomorphisms.  The
endofunctor $w_{15}$ sends $A$ to $w_{15}(A)$, which has the same
underlying abelian scheme; the action is precomposed with the
automorphism of $\Lambda$ given by conjugation by $xy \in D$.  (Note
that the $\Lambda$-linear map $xy\co A \to w_{15}(A)$ factors through
an isomorphism $A/A[xy] \to w_{15}(A)$, between $w_{15}(A)$ and the
quotient of $A/A[xy]$ by the subgroup of points annihilated by $xy$.)
We have $(w_{15})^2 = Id$.  The stack ${\cal X}^D/w_{15}$ is obtained
by stackifying the functor whose $Y$-points form a groupoid with
the same objects as those of ${\cal X}^D$; the set of morphisms from
$A$ to $B$ is the disjoint union of the set of isomorphisms $A \to B$
and the set of isomorphisms $w_{15} A \to B$.

We choose a $2$-adic square root of $-15$.  The $2$-divisible group of
${\cal X}^D$ takes a $Y$-point $A$ to the $2$-divisible group
$A[2^\infty]$.  To extend this $2$-divisible group to ${\cal
  X}^D/w_{15}$, for an isomorphism $w_{15} A \to B$ we assign
$\frac{1}{\sqrt{-15}}$ times the composite isomorphism
\[
A[2^\infty] \longoverto^{xy} (w_{15} A)[2^\infty] \to B[2^\infty].
\]
To verify that this preserves composition, we must check that it
commutes with $w_{15}$, which is straightforward.
\end{proof}

\begin{prop}
\label{prop:w15-action}
The involution on the ring of automorphic forms induced by $w_{15}$
acts trivially on the generators $b_1^2$, $b_1 b_3$, and $b_3^2$, and
sends $b_6$ to $-b_6-(b_1^6+b_1^3b_3+b_3^2)$.
\end{prop}

\begin{proof}
We can apply the Eichler-Selberg trace formula (as in
\cite[Section~3.5]{shimc}) to determine the action of the Atkin-Lehner
operators on the ring of automorphic forms.  It shows that the ring
endomorphism associated to $w_{15}$ has trace $-15$ in degree $2$, and
$2(-15)^{k}$ in degrees $2k > 2$.

The induced involution which gives rise to the $2$-divisible group on
the quotient stack ${\cal X}^D/w_{15}$ rescales the action on forms in
weight $k$ by $({\sqrt{-15}})^{-k}$; it therefore has trace $1$ in
degree $2$ and $2$ in higher degrees.

The form $b_1^2$ generates all forms of weight $2$, and is thus
fixed; the forms $b_1^4$ and $b_1 b_3$ generate all forms of weight
$4$, and are fixed as well.  Since $b_3^2 b_1^2 = (b_1 b_3)^2$ and
$b_1^2$ is not a zero divisor, $b_3^2$ must also be fixed.

As the forms $b_1^6$, $b_3 b_1^3$, and $b_3^2$ are all fixed, the
remaining generator $b_6$ in degree $6$ must be sent to $-b_6$ plus a
fixed element.  The result must satisfy the same minimal polynomial
as $b_6$, so it is forced to map to the conjugate root
$-b_6-(b_1^6+b_1^3b_3+b_3^2)$ as desired.
\end{proof}

\begin{prop}
The involution $w_{15}$ lifts to an involution on the double cover
${\cal Y}$, along with its $2$-divisible group, that fixes the forms
$a_1$ and $a_3$ and commutes with the $\mb Z/2$-action on ${\cal Y}$.
\end{prop}

\begin{proof}
By Remark~\ref{rmk:torsionorbits}, the points of ${\cal Y}$ are
locally described as abelian schemes $A$ equipped
with an equivalence class $[p]$ of $xy$-torsion point under the
action of a subgroup of $H < (\Lambda/(xy))^\times$.  Conjugation by
$xy$ preserves this subgroup, and so we can lift the involution
$w_{15}$ to an involution $(A,[p]) \mapsto (w_{15} A, [p])$ of
${\cal Y}$. This is an automorphism over ${\cal O}_K$ and commutes
with the $\mb Z/2$-action, which sends $[p]$ to $[kp]$ for some $k \in
\Lambda$.

This lift of $w_{15}$ fixes $b_1^2 = x a_1^2$, and so must either fix
$a_1$ or negate it; however, by possibly composing with the $\mb
Z/2$-action on ${\cal Y}$ we can replace it with a lift fixing $a_1$.
It then fixes $a_1$ and $a_1 a_3$, and since $a_1$ is not a zero
divisor it fixes $a_3$.
\end{proof}

\section{Spectra}

As in \cite[2.6]{shimc}, the $2$-divisible group on the
Deligne-Mumford stack ${\cal X}^D/w_{15}$ gives rise to a sheaf ${\cal
O}^{der}$ of spectra on the \'etale site of $\comp{({\cal
X}^D/w_{15})}_2$.  In particular, we can define global section objects:
\begin{align*}
E &= \Gamma(\comp{{\cal Y}}_2, {\cal O}^{der})\\
\TAF^D &= \Gamma(\comp{({\cal X}^D)}_2, {\cal O}^{der})\\
E^{w_{15}} &= \Gamma(\comp{({{\cal Y}/w_{15}})}_2, {\cal O}^{der})\\
(\TAF^D)^{w_{15}} &= \Gamma(\comp{({\cal X}^D/w_{15})}_2, {\cal O}^{der})
\end{align*}
The expressions of ${\cal X}^D$, ${\cal Y}/w_{15}$, and ${\cal
  X}^D/w_{15}$ as quotients of ${\cal Y}$ by group actions give us
fixed-point expressions:
\begin{align*}
\TAF^D &\simeq E^{h(\mb Z/8 \times 1)}\\
E^{w_{15}} &\simeq E^{h(1 \times \mb Z/2)}\\
(\TAF^D)^{w_{15}} &\simeq E^{h(\mb Z/8 \times \mb Z/2)}
\end{align*}
Our goal in this section is to analyze the resulting homotopy
fixed-point spectral sequences.  We note that the $\mb Z/8
\times \mb Z/2$-invariant cover of ${\cal Y}$ by $a_1^{-1} {\cal Y}$
and $a_3^{-1} {\cal Y}$ gives rise to an equivariant homotopy pullback
diagram:
\begin{equation}
  \label{eq:mayervietoris}
\xymatrix{
E \ar[r] \ar[d] & a_1^{-1} E \ar[d] \\
a_3^{-1} E \ar[r] & (a_1 a_3)^{-1} E
}
\end{equation}
On homotopy groups, this determines a square which
is part of a Mayer-Vietoris sequence:
\[
\xymatrix{
\pi_* E \ar[r] \ar[d] & {\cal O}_K[a_1^{\pm 1},a_3,a_6]/f(a_i) \ar[d] \\
{\cal O}_K[a_1,a_3^{\pm 1},a_6]/f(a_i) \ar[r] & 
{\cal O}_K[a_1^{\pm 1},a_3^{\pm 1},a_6]/f(a_i)
}
\]

\begin{prop}
\label{prop:conncover}
There is an $E_\infty$ ring spectrum $e$ with a $\mb Z/8 \times \mb
Z/2$-equivariant map $e \to E$ which, on homotopy groups, is the inclusion
\[
\mb {\cal O}_K[a_1, a_3, a_6]/f(a_i) \to \pi_* E
\]
of terms in even total degree which were contributed by $H^0({\cal Y})$.
\end{prop}

\begin{proof}
Let $\tilde e$ be the connective cover of $E$; the map $\tilde e \to
E$ is a $\mb Z/8 \times \mb Z/2$-equivariant map, and has the desired
behavior on homotopy groups except that $\pi_1 \tilde e$ is ${\cal O}_K
\neq 0$.  Let $P_1 \tilde e = \tilde e[0,1]$ be the Postnikov stage of
$\tilde e$.  As the element $\eta \in \pi_1(\comp{\mb S}_2)$ maps to
zero in $\pi_1(\tilde e)$, the map $\comp{\mb S}_2 \to P_1 \tilde e$
factors through the Postnikov stage of the equivariant $E_\infty$ cone
on $\eta$, which is $H\mb Z_2$. Similarly, attaching cells to $H\mb
Z_2$ allows us to extend to a map of $E_\infty$ rings $H{\cal O}_K \to
P_1 e$. We can therefore form the homotopy pullback
\[
\xymatrix{
e \ar[r] \ar[d] & \tilde e \ar[d] \\
H{\cal O}_K \ar[r] & P_1 \tilde e,
}
\]
which has the desired properties.
\end{proof}

A direct consideration of homotopy groups gives the following.
\begin{cor}
The maps $a_1^{-1} e \to a_1^{-1} E$ and $a_3^{-1} e \to a_3^{-1} E$
are equivalences.
\end{cor}

\section{Homotopy groups of $\TAF^D$}

\subsection{Connective fixed points}
We begin to compute the homotopy groups of $\TAF^D$ by analyzing the
homotopy fixed point spectral sequence for the spectrum $e$ in
Proposition~\ref{prop:conncover}.  These carry a very strong
resemblance to the calculations of \cite{mahowald-rezk-level3}, with
the addition of the generator $b_6$ which operates largely
independently.

We will write $\zeta$ for the generating element of $H^1(\mb Z/2; \mb
Z_2^{sgn})$.  This allows us to give the following short expression of
the group cohomology.

\begin{prop}
  The $E_2$-term of the homotopy fixed point spectral sequence for
  $\mb Z/8$ acting on $e$ takes the following form:
\[
    E_2^{s,t} = \mb Z_2[b_1^2, b_1 b_3, b_3^2, b_6, \zeta^2, \zeta a_1,
    \zeta a_3] / (2\zeta, f(b_i))
\]
\end{prop}
In particular, these terms are concentrated in degrees with $(t-s)
\equiv s$ mod $4$, and the only possible differentials are $d_{4k+3}$.

\begin{rmk}
As $x$ reduces to $1$ mod $2$, $\zeta x = 0$, and hence for any
polynomial $g(a_1,a_3, a_6)$ which is homogeneous of even degree we
have $\zeta g(a_1, a_3, a_6) = \zeta g(b_1, b_3, b_6)$. This allows us
to abuse notation by writing $\zeta b_i = \zeta a_i$.
\end{rmk}

\begin{prop}
\label{prop:detection}
Under the map $\pi_* \mb S \to e^{h\mb Z/2}$, the images of $\eta \in
\pi_1 \mb S$ and $\nu \in \pi_3 \mb S$ are represented by $\zeta b_1$
and $\zeta^3 b_3$ in the homotopy fixed point spectral sequence.
\end{prop}
\begin{proof}
Completion at the supersingular locus gives a map $e \to E_2$ to a
Lubin-Tate spectrum over $\mb F_{16}$, which is expressed on homotopy
groups
\[
\mb Z_2[b_1, b_3, b_6] / f(b_i) \to {\cal O}_K\pow{u_1} [u^{\pm 1}]
\]
by sending $b_1$ to $u_1 u$, $b_3$ to $u^3$, and $b_6$ to the unique
root of $f(b_i)$ which reduces, mod $(2,u_1)$, to $\omega u^6$, where
$\omega$ is a chosen primitive third root of unity.  This expresses
$L_{K(2)} e$ as a homotopy fixed point spectrum of $E_2$ by an action of
$\mb Z/3$, and the $\mb Z/2$-action on $e$ realizes to the action of
$[-1]$ as a subgroup of the Morava stabilizer.  The elements $\eta$
and $\nu$ are detected on the $1$-line and $3$-line respectively of
the associated homotopy fixed point spectral sequence for $E_2$, by a
transfer argument in the latter case\cite{mahowald-rezk-level3}. They
are hence also detected in the homotopy fixed point spectral sequence
for $e$.
\end{proof}

\begin{prop}
In the homotopy fixed point spectral sequence, we have differentials
\begin{align*}
d_3(\zeta^2) &= \zeta^5 b_1\text{ and}\\
d_7(\zeta^4) &= \zeta^{11} b_3.
\end{align*}
\end{prop}
\begin{proof}
In the homotopy fixed point spectral sequence for the action of $\mb
Z/2$ on $\mb S$, which coincides with the Atiyah-Hirzebruch spectral
sequence for the stable cohomotopy of $\mb{RP}^\infty$, the cell
attachment structure of $\mb{RP}^\infty$ implies that we have
differentials
\begin{align*}
d_2(\zeta^2) &= \zeta^4 \eta\text{, and}\\
d_4(\zeta^4) &= \zeta^8 \nu.
\end{align*}
We truncate to a skeleton of $E\mb Z/2$ and compare this with the
homotopy fixed-point spectral sequence for $e$:
\[
F(\mb{RP}^k, \mb S) \to F_{\mb Z/2}((E\mb Z/2)^{(k)}, e).
\]
When $k=5$, the differentials we described show that $\zeta^4 \eta$ is
either trivial or detected by $\zeta^5 \eta^2$ in the spectral
sequence calculating $\pi_{-4} F(\mb{RP}^5, \mb S)$. Therefore,
$\zeta^5 b_1$ is equivalent to a class which could only be detected by
$0$ or the class $\zeta^7 b_1^2 = 0$
in the spectral sequence calculating $\pi_* F_{\mb Z/2}((E\mb
Z/2)^{(5)}, e)$, and thus must be the target of a differential; the
only possible source is $\zeta^2$.

Similarly, when $k=11$ the differentials show that $\zeta^8 \nu$ is
either trivial or detected by $\zeta^{11} \nu$; in the spectral
sequence calculating $\pi_* F_{\mb Z/2}((E\mb Z/2)^{(5)}, e)$ the
element $\zeta^{11} b_3$ is equivalent to a class that could only be
detected by zero or $\zeta^{14} b_3 = 0$. It then must be the target
of a differential; the only possible source after the $E_2$-page is
$\zeta^4$.
\end{proof}

\begin{prop}
The $d_3$-differentials in the homotopy fixed-point spectral sequence
are determined by the following differentials and the Leibniz rule:
\begin{align*}
  d_3(\zeta^2) &= \zeta^5 b_1 &
  d_3(b_1^2) &= \zeta^3 b_1^3 &
  d_3(b_3^2) &= \zeta^3 b_1 b_3^2 &
  d_3(b_1 b_3) &= 0 \\
  d_3(\zeta b_1) &= 0 &
  d_3(\zeta b_3) &= \zeta^4 b_1 b_3 &
  d_3(b_6) &= \zeta^3 b_1 b_6
\end{align*}
\end{prop}

\begin{proof}
The differential on $\zeta^2$ is determined by the previous
proposition.

For a class in even total degree which is negated by the $\mb
Z/2$-action, \cite{mahowald-rezk-level3} describes a cup-one identity
$d_3(x^2) = \eta (\zeta x)^2$ in the homotopy fixed-point spectral
sequence for the action of the subgroup $\mb Z/2$, which determines
a nontrivial differential
\[
d_3(b_1^2) = d_3(x a_1^2) = x \zeta^3 b_1^3 = \zeta^3 b_1^3.
\]
Naturality of the spectral sequence implies that there is a
corresponding differential in the homotopy fixed-point spectral
sequence for the action of $\mb Z/8$. The differential on $b_3^2$
follows similarly.

The elements $\zeta b_1$ and $\zeta^3 b_3$ support no differentials by
Proposition~\ref{prop:detection}.

The Leibniz rule shows the identity
\[
0 = d_3((\zeta b_1) (\zeta^3 b_3)) = \zeta^4 d_3(b_1 b_3).
\]
Similarly, we have
\[
0 = d_3(\zeta^3 b_3) = \zeta^6 b_1 b_3 + \zeta^2 d_3(\zeta b_3).
\]
Multiplication by $\zeta^2$ induce isomorphisms from the $s$-line to the
$s+2$-line for $s > 0$, and so this forces the identities $d_3(b_1 b_3)
= 0$ and $d_3(\zeta b_3) = \zeta^4 b_1 b_3$.

Finally, applying $d_3$ to the identity $f(b_i) = 0$ gives rise to the
identity
\[
0 = d_3(b_6) (b_1^6 + b_1^3 b_3 + b_3^2) + (b_1^6 + b_1^3 b_3 +
b_3^2) \zeta^3 b_1 b_6,
\]
which forces the desired differential on $b_6$.
\end{proof}

It is convenient to write a class on the zero-line in the form
$b_1^{4k - 3l} b_3^l$ or $b_1^{2 + 4k - 3l} b_3^l b_6$ if it is in degree
congruent to $0$ mod $8$, and in the form $b_1^{-2 + 4k - 3l} b_3^{l}$
or $b_1^{4k - 3l} b_3^l b_6$ if it is in degree equivalent 
to $4$ mod $8$.  In these terms we have
\begin{align*}
d_3(b_1^{2 + 4k - 3l} b_3^l) &= \eta^3 \cdot b_1^{4k - 3l} b_3^l,\\
d_3(b_1^{4(k+1) - 3l} b_3^l b_6) &= \eta^3 \cdot b_1^{2 + 4k - 3l} b_3^l b_6.
\end{align*}
More, we have differentials
\begin{align*}
d_3(\zeta^{4k+2}) &= \eta \cdot \zeta^{4k+4}\text{ and}\\
d_3(\zeta^{4k} b_6) &= \eta \cdot \zeta^{4k+2} b_6.
\end{align*}
Mod $2$, we then find $d_3$ is injective on classes where $t - 2s \equiv 4$
mod $8$.

A schematic diagram of this differential appears in
Figures~\ref{fig:e3-page} and \ref{fig:e3-page-a6}, broken up
according to whether the classes are multiples of $b_6$.  Note that an
arrow indicates that {\em all} classes in the indicated degree support
a differential.
\begin{figure}[h]
\centerline{\includegraphics[width=5in]{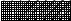}}
\caption{$E_3$-page, classes which are not multiples of $b_6$}
\label{fig:e3-page}
\end{figure}
\begin{figure}[h]
\centerline{\includegraphics[width=5in]{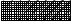}}
\caption{$E_3$-page, classes which are multiples of $b_6$}
\label{fig:e3-page-a6}
\end{figure}

The surviving classes on the zero-line are generated by
\[
b_1^4, b_1 b_3, b_3^4, b_1^2 b_6, b_3^2 b_6, 2b_1^2, 2b_3^2, 2b_6;
\]
the classes which are not multiples of $2$ support $\eta^2$-multiples.
In filtrations above $2$, almost all classes are annihilated.  The
remaining ones are all $\nu$-multiples of the classes $\eta^k
\zeta^{4l}$ and $\eta^k \zeta^{4l+2} b_6$ for $0 \leq k \leq 2$,
$\zeta^l b_3^{4k-l}$, and $\zeta^l b_3^{4k-l+2} b_6$.  The remaining
classes on the $1$-line and $2$-line consist only of the
$b_3^4$-multiples of the following classes:
\[
\zeta^2 b_3^2, \zeta b_3^3, \eta \zeta b_3^3, \zeta b_3 b_6, \zeta^2
b_6, \eta \zeta b_3 b_6
\]

Multiplication by $b_3^4$, $\zeta^4$, and $\nu$ are injective in
filtrations greater than or equal to three.
\begin{prop}
The only nonzero $d_7$-differentials in the homotopy fixed-point
spectral sequence are
\begin{align*}
  d_7(\zeta^k b_3^l) &= \zeta^{k+7} b_3^{l+2}
  &\text{ if }3l - k \equiv 4 \mod 8,\text{ and}\\
  d_7(\zeta^k b_3^l b_6) &= \zeta^{k+7} b_3^{l+2} b_6
  &\text{ if }3l-k \equiv 6 \mod 8.
\end{align*}
\end{prop}

This differential is pictured in Figures~\ref{fig:e7-page} and
\ref{fig:e7-page-a6}; dashed lines indicate multiplication by $\nu$.
It may be more concisely expressed as saying that $\nu$-towers in degrees
$t-2s \equiv 8$ mod $16$ support differentials which hit the
$\nu$-towers in degrees $t-2s \equiv 0$ mod $16$.
\begin{figure}[h]
\centerline{\includegraphics[width=5in]{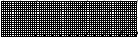}}
\caption{$E_7$-page, classes which are not multiples of $b_6$}
\label{fig:e7-page}
\end{figure}
\begin{figure}[h]
\centerline{\includegraphics[width=5in]{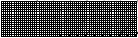}}
\caption{$E_7$-page, classes which are multiples of $b_6$}
\label{fig:e7-page-a6}
\end{figure}

\begin{proof}
We first note that all possible targets of a differential are elements
which are not $\nu$-torsion, and so the only possible sources of a
differential are the non-$\nu$-torsion elements $\zeta^k b_3^l$ and
$\zeta^k b_3^l b_6$ which have survived to $E_7$.

We have differentials $d_7 \zeta^4 = \nu \zeta^8$, and $d_7 \nu = 0$;
by the Leibniz rule, we then find the desired differentials on all classes of
the form $\zeta^k b_3^l$.

Multiplying the equation $f(b_i) = 0$ by $\zeta^4$ gives the identity
on $E_7$-pages
\[
(\zeta^2 b_6)^2 + (\zeta^2 b_3^2) (\zeta^2 b_6) + (\zeta^2 b_3^2)^2 = 0.
\]
The Leibniz rule then implies $d_7(\zeta^2 b_3^2) (\zeta^2 b_6) =
\zeta^2 b_3^2 d_7(\zeta^2 b_6)$, and thus $d_7(\zeta^2 b_6) = \nu
\zeta^4 (\zeta^2 b_6)$.  The Leibniz rule then determines the
differentials on all elements $\zeta^k b_3^l b_6$.
\end{proof}

There is no room for further differentials; the only remaining classes
in high filtration are $\zeta^{8k}$ and $\zeta^{8k+6} b_6$, which are
in even total degree and cannot be a target of a differential.

The following assembles the final result.
\begin{thm}
  The nonnegative-degree homotopy groups of $e^{h\mb Z/2}$ fit into a
  short exact sequence $0 \to K \to \pi_* e^{h\mb Z/2}\langle
  0\rangle \to R \to 0$. The terms in this sequence are given as
  follows.
  \begin{itemize}
  \item $R$ is the pullback in the diagram of rings
\[
\xymatrix{
R \ar[r] \ar[d] &
\mb Z[b_1^2, b_3^2, b_6, \eta] / (2\eta, \eta^3, f(b_i)) \ar[d] \\
\langle b_1^4, b_1 b_3, b_1 b_3^5, b_3^8, b_1^2 b_6, b_1 b_3^3 b_6,
b_3^6 b_6, \eta \rangle \ar@{^(->}[r]&
\mb Z/2[b_1^2, b_1 b_3, b_3^2, b_6, \eta] / (2\eta, \eta^3, f(b_i)),
}
\]
where the lower map is the inclusion of the subring generated by the given
elements.
  \item $K$ is freely generated over $\mb Z/2[b_3^8]$ by the following
    classes:
\begin{align}
\label{eq:spurious-classes}
&\nu, \nu^2,
&&\nu (b_3^6 b_6), \nu^2 (b_3^6 b_6), \notag\\
&\zeta b_3^3, \eta (\zeta b_3^3),\nu (\zeta b_3^3),
&&\zeta b_3 b_6, \eta (\zeta b_3 b_6),\nu (\zeta b_3 b_6),\\
&\eta (b_3^4), \eta^2 (b_3^4),
&&\eta (b_3^2 b_6), \eta^2 (b_3^2 b_6),\notag\\
&\zeta^2 b_3^6, \nu (\zeta^2 b_3^6),
&&\zeta^2 b_3^4 b_6, \nu (\zeta^2 b_3^4 b_6).\notag
\end{align}
  \end{itemize}
\end{thm}

\subsection{Nonconnective fixed points}

We can use the results of the previous section to determine the
homotopy fixed point spectrum for the action of $\mb Z/2$ on $E$.

Due to the size and shape of the vanishing regions in the homotopy
fixed point spectral sequences of the previous section, we observe
that the localizations $(a_1^{-1} e)^{h\mb Z/2}$ and $(a_3^{-1}
e)^{h\mb Z/2}$ have homotopy groups which are calculated by the
localizations of the $E_\infty$ terms of the homotopy fixed-point
spectral sequences.  We write these as $b_1^{-1} \TAF^D$ and
$b_3^{-1} \TAF^D$, though they are actually formed by inverting
$b_1^4$ and $b_3^8$.  The elements $\zeta^{8k}$ and $\zeta^{8k+6} b_6$
are destroyed in these localizations.

The homotopy groups of $b_1^{-1} \TAF^D$ form the ring which fits into
the following pullback diagram.
\[
\xymatrix{
\pi_* b_1^{-1} \TAF^D  \ar[r] \ar[d] &
\mb Z[b_1^{\pm 2}, b_1 b_3, b_6, \eta] / (2\eta, \eta^3, f(b_i)) \ar[d] \\
\langle b_1^{\pm 4}, b_1 b_3, b_1^2 b_6, \eta \rangle \ar@{^(->}[r]&
\mb Z/2[b_1^{\pm 2}, b_1 b_3, b_6, \eta] / (2\eta, \eta^3, f(b_i)),
}
\]

The homotopy groups of $b_3^{-1} \TAF^D$ are more convenient to study
by introducing $y = b_3^{-2} b_6$, which is a degree zero term
satisfying a quadratic polynomial.  The homotopy groups are free on
$\{1, y\}$ over a subring $S$ which essentially coincides with the
calculation of Mahowald-Rezk \cite{mahowald-rezk-level3}.
This subring lives in a short exact sequence $0 \to L \to
S \to b_3^{-1} S \to 0$.  Here $b_3^{-1} S$ fits into
the pullback diagram
\[
\xymatrix{
b_3^{-1} S \ar[r] \ar[d] &
\mb Z[b_1^2, b_3^{\pm 2}, \eta] / (2\eta, \eta^3) \ar[d] \\
\langle b_1 b_3, b_1 b_3^5, b_3^{\pm 8}, \eta \rangle \ar@{^(->}[r]&
\mb Z/2[b_1^2, b_1 b_3, b_3^{\pm 2}, \eta] / (2\eta, \eta^3),
}
\]
and $L$ is a free $\mb Z/2[b_3^{\pm 8}]$-module on the classes from
equation~(\ref{eq:spurious-classes}) that do not involve $b_6$.

The long exact sequence on homotopy groups induced by the homotopy
pullback can then be employed to produce a description of the homotopy
groups of $\TAF^D$.  There is a short exact sequence $0 \to b_3^{-1} K
\to \pi_* \TAF^D \to R' \to 0$, where the ring $R'$ agrees with $R$ in
nonnegative degrees.

\section{Homotopy groups of $(\TAF^D)^{w_{15}}$}

In this section we compute the homotopy groups of $\TAF^D$ after
taking homotopy fixed points with respect to the Atkin-Lehner operator
$w_{15}$.  Much like the previous section, this largely coincides with
the Mahowald-Rezk computation, but adds some new $v_1$-periodic
classes.

We first consider the homotopy fixed-point spectral sequence for the
spectrum $e^{w_{15}}$; here $e$ is the connective spectrum of
Proposition~\ref{prop:conncover}. Recall that $\omega \in {\cal O}_K$
is a primitive third root of unity.

\begin{prop}
The homotopy fixed point spectral sequence for the homotopy groups of
$e^{w_{15}}$ degenerates at $E_2 = E_\infty$, with target
\begin{equation}
  \label{eq:w15-connective}
{\cal O}_K[a_1,a_3,\tau] / (2\tau, (a_1^3 - \omega a_3)(a_1^3 -
\omega^2 a_3)\tau).
\end{equation}
The homotopy fixed point spectral sequences for $a_1^{-1} e$,
$a_3^{-1} e$, and $(a_1 a_3)^{-1} e$ degenerate in
the same way.
\end{prop}
\begin{proof}
By Proposition~\ref{prop:w15-action}, the action of $w_{15}$ fixes
$a_1$ and $a_3$, but sends $a_6$ to $-a_6-(a_1^6 + a_1^3 a_3 +
a_3^2)$.  A direct calculation of the group cohomology gives the
$E_2$-term described in equation~(\ref{eq:w15-connective}), where
$\tau$ is the generator of $H^2(\mb Z/2; \mb Z)$.  The terms are
concentrated in even degrees, and so the spectral sequence collapses.

The corresponding calculation for the localizations follows in the
same fashion.
\end{proof}

\begin{prop}
\label{prop:twoseries}
The class $\tau$ of equation~(\ref{eq:w15-connective}) lifts to a
class in the homotopy of $e^{w_{15}}$ satisfying $[2](\tau) = 0$,
where $[2](x)$ is the $2$-series of the formal group law of $e$.
\end{prop}

\begin{proof}
The spectrum $e^{w_{15}}$, being a ring spectrum with homotopy concentrated in
even degrees, is complex orientable; we may recast the resulting map
$MU \to e^{w_{15}}$ as a $\mb Z/2$-equivariant map $MU \to e$, where
$\mb Z/2$ acts trivially on $MU$ and by the Atkin-Lehner involution
$w_{15}$ on $e$.  The resulting map on homotopy fixed points
is a map $MU^*(B\mb Z/2) \to \pi_* e^{w_{15}}$, and the class $\tau$
lifts to the orientation in $MU^2(\mb{CP}^\infty)$ whose image in
$MU^2(B\mb Z/2)$ satisfies $[2](\tau) = 0$.
\end{proof}

\begin{cor}
\label{cor:degeneration}
The class $\tau$ lifts to an element in the homotopy of the spectrum
$E^{w_{15}}$ on which $a_1$ and $a_3$ act invertibly.
\end{cor}

\begin{proof}
We have that $(a_1^6 + a_1^3 a_3 + a_3^2)\tau$ is zero in the homotopy
fixed point spectral sequences.  After inverting $a_3$, the action of the
element $a_1^3/a_3$ is as a third root of unity on $\tau$, and hence
invertible.  The element $a_3/a_1^3$ acts as a third root of unity
after inverting $a_1$.  As the action of $a_1$ and $a_3$ on $\tau$ are
invertible after inverting any of $a_1$, $a_3$, or $a_1 a_3$, it is so
in $E^{w_{15}}$.
\end{proof}

\begin{cor}
We have an isomorphism
\[
\pi_* E^{w_{15}} \cong (E_*)^{w_{15}} \oplus ({\cal O}_K[a_1^{\pm 1}])^2 \cdot \tau,
\]
We have $\tau^2 = \alpha \tau$ for $\alpha$ some element congruent to
$2a_1^{-1}$ mod $4$.
\end{cor}

\begin{proof}
By Corollary~\ref{cor:degeneration}, the action of $a_1$ on $\tau$ is
invertible, and by Proposition~\ref{prop:twoseries} we know 
$\tau$ is annihilated by the two-series.  We have
\[
[2](\tau) = 2\tau(1 + \tau^2 f_0(\tau^2)) + a_1 \tau^2(1 + \tau^2 f_1(\tau^2))
\]
for some power series $f_0$ and $f_1$, and this implies
\[
\tau^2 = a_1^{-1}(-2\tau)(1 + \tau^2 g(\tau^2))
\]
for some power series $g$.  Applying the Weierstrass preparation
theorem to $g$, we find that $\tau^2 = \alpha \tau$ for some $\alpha$
congruent to $2a_1^{-1}$ mod $4$.

Mod $2$, the element $a_3/a_1^3$ acts as a third root of unity on
$\tau$.  By Hensel's lemma, there is a $2$-adically convergent power
series which is a unit multiple of $a_3/a_1^3$, acting as a third root
of unity on $\tau$. This determines the action of the entire ring
$(a_1^{-1}E_*)^{w_{15}}$: in particular, it factors through ${\cal
  O}_K[a_1^{\pm 1}](\omega) \cong ({\cal O}_K[a_1^{\pm 1}])^2$.
\end{proof}

We now consider the homotopy fixed point spectral sequence for the
action of $\mb Z/8$ on $E^{w_{15}}$.  We have a comparison with the
homotopy fixed point spectral sequence for $E$, and we will now show
that $\tau$ must support some nontrivial differential.

\begin{prop}
\label{prop:notau}
The image of the class $\tau$ in the homotopy of $E^{w_{15}}$ does not
lift to the homotopy of $(\TAF^D)^{w_{15}}$, or its
localizations by $a_1$, $a_3$, or $(a_1 a_3)$.
\end{prop}

\begin{proof}
To show that $\tau$ does not lift to the homotopy fixed points of $\mb
Z/8 \times \mb Z/2$ on $E$, we consider the $\mb Z/2$-equivariant maps
$\mb S \to \TAF^D \to E$, where the Atkin-Lehner involution acts
on both $\TAF^D$ and $E$.  The class $\tau$ appears as the
unique generating class in filtration $2$ and homotopy degree $-2$ in
all three homotopy fixed point spectral sequences, and in the homotopy
fixed point spectral sequence for $\TAF^D$ the only nonzero class in
higher filtration is $\zeta^{14} b_6$.

The homotopy fixed point spectral sequence for $\mb S$ calculates the
stable cohomotopy groups of $\mb{RP}^\infty$, and in this spectral
sequence the class $\tau$ supports a $d_2$ differential $d_2(\tau) =
\eta \tau^2$ due to the attaching maps for the cells constructing
$\mb{RP}^\infty$.  As the map $\mb S \to \TAF^D$ is an inclusion
of a summand in homotopy degrees $-7$ through $2$, the class $\tau$
also supports a $d_2$ differential in the homotopy fixed point
spectral sequence for $(\TAF^D)^{w_{15}}$.

The only class in higher filtration whose image could possibly be
$\tau$ is $\zeta^{14} b_6$.  However, this class is annihihilated by
inverting $a_1$, which acts invertibly on $\tau$.
\end{proof}

\begin{prop}
In the homotopy fixed point spectral sequence for the action of $\mb
Z/8$ on $E^{w_{15}}$, we have a $d_3$-differential $d_3(\tau) =
\eta^3 b_1^{-2} \tau$.  All remaining differentials are determined
by their image in the homotopy fixed point spectral sequence for
$E$.
\end{prop}

\begin{proof}
We first consider the $E_3$-page.  The elements $\tau$ and $b_1^2
\tau$ are in the kernel of the map to the $E_3$-page for $E$, and so
can only support differentials to unit multiples of $\eta^3 b_1^{-2}
\tau$ or $\eta^3 \tau$ respectively.

From the Leibniz rule either $\tau$ or $b_1^2 \tau$ must support a
nonzero differential.  As the element $b_3/b_1^3$ is a cycle, we cannot
have both of these differentials occuring.  No matter which
differential occurs, no elements in the kernel of the map of homotopy
fixed-point spectral sequences survive above the $2$-line, and so this
is the only possible differential that $\tau$ could support.
Therefore, we must have $d_3(\tau) = \eta^3 b_1^{-2} \tau$ and
$d_3(b_1^2 \tau) = 0$.

The $d_3$ differential is determined on classes which are
not multiples of $\tau$ by their image in the homotopy fixed point
spectral sequence for $E$.  On the $E_4$-page, the homotopy fixed
point spectral sequence for $E^{w_{15}}$ maps surjectively onto that
for $E$, and the map is an isomorphism above the $2$-line.  Therefore,
the remaining differentials are determined by their image in the
homotopy fixed point spectral sequence for $E$.
\end{proof}

As a consequence of this description of the spectral sequence, we have
the following.
\begin{thm}
Under the map
\[
\pi_* (\TAF^D)^{w_{15}} \to \pi_* \TAF^D,
\]
the kernel is isomorphic to $\pi_*(\Sigma^2  KO) \otimes W(\mb F_4)$, and
the image in the homotopy of $(\TAF^D)^{w_{15}}$ consists of those
classes in the homotopy of $\TAF^D$ which are not multiples of $b_6$.
\end{thm}

\bibliography{masterbib}

\end{document}